\documentclass{amsart}
\usepackage{amsaddr}

\usepackage{amssymb,amsmath,amsthm}
\usepackage{hyperref}
\usepackage{mathabx}
\usepackage[all]{xy}
\usepackage[utf8]{inputenc}
\usepackage[T1]{fontenc}
\usepackage{lmodern}
\usepackage{ mathrsfs }
\usepackage{ dsfont }
\usepackage [american]{babel}
\usepackage{textcomp}

\usepackage{enumerate}

\newcommand{\Kl}[1]{\left( #1 \right)}

\newcommand{\Z}{{\mathbb Z}}
\newcommand{\N}{{\mathbb N}}
\newcommand{\Q}{{\mathbb Q}}
\newcommand{\R}{{\mathbb R}}

\newcommand{\Ha}{{\mathbb H}}

\newcommand{\eps}{{\epsilon}}

\newcommand{\RN}[1]{\uppercase\expandafter{\romannumeral#1}}

\usepackage{comment}

\newtheorem{thm}{Theorem}
\newtheorem{lemma}[thm]{Lemma}

\newtheorem{kor}[thm]{Corollary}

\theoremstyle{definition}

\newcommand{\secret}[1]{}

\newcommand{\pdiv}{\mid\!\mid}

\usepackage{mathtools}
\usepackage{graphicx} 

\DeclarePairedDelimiter\abs{\lvert}{\rvert}%
\DeclarePairedDelimiter\norm{\lVert}{\rVert}%

\makeatletter
\let\oldabs\abs
\def\abs{\@ifstar{\oldabs}{\oldabs*}}
\let\oldnorm\norm
\def\norm{\@ifstar{\oldnorm}{\oldnorm*}}

\renewcommand*\env@cases[1][1.2]{%
	\let\@ifnextchar\new@ifnextchar
	\left\lbrace
	\def\arraystretch{#1}%
	\array{@{}l@{\quad}l@{}}%
}

\makeatother

\title{Uniform bounds for norms of theta series and arithmetic applications}
\author{Fabian Waibel}

\begin{document}

 \maketitle
 
 \vspace{-0.6cm}
 
  \begin{abstract} \noindent	We prove uniform bounds for the Petersson norm of the cuspidal part of the theta series. This gives an improved asymptotic formula for  the number of representations by a quadratic form. As an application, we show that every integer $n \neq 0,4,7 \,(\operatorname{mod}8)$ is represented as $n= x_1^2 + x_2^2 + x_3^3$ for integers $x_1,x_2,x_3$ such that the product $x_1x_2x_3$ has at most 72 prime divisors. 
\end{abstract}

\renewcommand{\thefootnote}{\fnsymbol{footnote}} 
\renewcommand{\thefootnote}{\arabic{footnote}}

\section{Introduction} 

\subsection{Quadratic forms and theta series} A positive integral $m{\times}m$ matrix $Q$ with even diagonal entries gives rise to a quadratic form  $q(x)= \frac{1}{2} x^T Q x$. It is one of the classical tasks of number theory to study which numbers $n$ are represented by $q$ or more precisely to count the number of solutions
\begin{align*}
r(Q,n) := \# \{ x \in \Z^{m} \, | \, \tfrac{1}{2} x^T Q x= n \}. 
\end{align*}
Naturally, $r(Q,n) \ge 1$ only holds if $n$ is locally presented by $q$ meaning that for every prime $p$ there exists a $p$-adic solution $x_p \in (\Z_p)^m$ of $n= q(x_p)$. By applying the Hardy-Littlewood method, Tartakowski \cite{Ta29} showed in 1929 that for $m \geq 5$  every locally represented, sufficiently large $n$ satisfies $r(Q,n) \ge 1$. Furthermore, he extended this result to $m=4$ by assuming that $n$ is primitively locally represented which means that there is a local solution for all primes such that at least one entry is a unit in $\Z_p$. 

To apply this result in practice, one needs an effective version stating how large $n$ has to be with respect to $Q$. For $m \ge 5$ various effective bounds were obtained by analytical methods, see for example Watson \cite{Wa1960}, Hsia and Icaza \cite{HI1999} and Browning and Dietmann \cite{Browning2007} and the references therein. However, these methods fail to provide satisfying results for $m=4$ or $m=3$, and one employs the theory of modular forms and theta series instead, as done by Blomer \cite{Bl2004}, Duke \cite{Duke2005}, Hanke \cite{Ha2004}, Rouse  \cite{Ro2019} or Schulze-Pillot \cite{SP2001}. 

To follow this latter approach, one considers 
\begin{align*}
\theta(Q,z) = \sum_{x \in \Z^m} e(\frac{1}{2}x^TQx z) = \sum_{n \geq 1} r(Q,n) e(nz)
\end{align*}
which is a modular form of weight $\frac{m}{2}$ with respect to 
\[\Gamma_0(N) = \Big\{ \begin{pmatrix}
a & b \\ c & d
\end{pmatrix} \in \operatorname{SL}_2(\Z) \mid c \equiv 0 \,(\operatorname{mod}N)\Big\},\]
where $N$ is the level of $Q$, and a quadratic character $\chi$. The idea is to approximate $\theta(Q,z)$ by a weighted average over the classes in the genus of $Q$. For this Eisenstein series, $\theta(\operatorname{gen}Q,n)$, the Fourier coefficients can be explicitly computed by
\begin{align} \label{eq:rgen}
r(\operatorname{gen} Q,n) = \frac{(2 \pi)^\frac{m}{2} n^{\frac{m}{2}-1} }{\Gamma(\frac{m}{2}) \sqrt{\det Q}} \prod_{p} \beta_p(n,Q),
\end{align}
where 
\begin{align} \label{eq:padicdensities}
\beta_p(n,Q) = \lim_{a \to \infty} p^{-a(m-1)} \# \Big\{x \in (\Z/p^a \Z)^{m} \, |\, \frac{1}{2} x^T Q x \equiv n \,(\operatorname{mod} p^{a}) \Big\}.
\end{align}

The $p$-adic densities $\beta_p(n,Q)$ contain the local restrictions. For $m \geq 4$ and $n$ primitively locally represented they satisfy ${\prod_p \beta_p(n,Q) \gg (n,N)^{-\delta} (nN)^{-\eps}}$, where ${\delta = \frac{1}{2}}$ for $m=4$ and $\delta =1 $ for $m \ge 5$. In many cases, this lower bound can be further improved, for details see Lemma \ref{lemma:lowerbound}.

The difference $f(z) = \theta(Q,z)- \theta(\operatorname{gen}Q,z)= \sum_n a(n)e(nz)$ is a cusp form. A uniform upper bound for $a(n)$ can be obtained by applying the Petersson formula, cf.\ \cite[Corollary 14.24]{Iw2004}. This gives for $m \geq 4$ that 
\begin{align} \label{eq:petersson}
a(n) \ll \norm{f} \, n^{\frac{m}{4}-\frac{1}{2}} \Big(1+ \frac{n^{\frac{1}{4}} (n,N)^{\frac{1}{4}}} {N^{\frac{1}{2}}}\Big) (nN)^\eps
\end{align} 
where the norm is induced by  the inner product
\begin{align*}
\langle f,g \rangle  = \int_{\Gamma_0(N) \backslash \Ha} f(z) \overline{g(z)} y^{\frac{m}{2}-2} dx\, dy.
\end{align*}


The case $m=3$ is more fragile and the saving in \eqref{eq:petersson} is too little. 
To circumvent this problem, one approximates $\theta(Q,z)$ by another  Eisenstein series $\theta(\operatorname{spin}Q,n)$ given as a weighted average over the classes in the spinor genus of $Q$. Then, the Fourier coefficients of $\theta(Q,z) - \theta(\operatorname{spin}Q,n)$ can be estimated by a method of Iwaniec-Duke, which uses the Shimura lift, the Kuznetsov formula and an elaborate estimate for sums over Kloosterman sums. To apply the resulting bound, \cite[Theorem 1]{Wa2017}, one employs a local argument by Blomer \cite[(1.7)]{Blomer2008} which shows that it is sufficient to bound $r(Q,n) - r(Q',n)$ for $Q,Q'$ in the same spinor genus and $(n,N)$ small. In many cases, $r(\operatorname{gen}Q,n)$ and $r(\operatorname{spn}Q,n)$ coincide, and one obtains an asymptotic formula for $r(Q,n)$. 

\subsection{Bounds for norms of theta series} The aim of this article is to give strong uniform bounds for the inner products of 
\begin{align*}
f(z) = \theta(Q,z) - \theta(\operatorname{gen}Q,z), \quad g(z)  = \theta(Q,z) - \theta(Q',z),
\end{align*}
where $Q,Q'$ are in the same genus. By definition of $\theta(\operatorname{gen}Q,z)$ any bound for $\langle g,g \rangle$ is automatically a bound for $\langle f,f \rangle$. 

The first bounds for $\langle f,f \rangle$ were only established  in the 1990's by Fomenko \cite{Fomenko1991} and Schulze-Pillot \cite{SP2001} for $m=4$. Stronger and more general estimates were proven by Blomer in his articles about ternary quadratic forms, \cite{Bl2004} and \cite{Blomer2008}. There, he shows for $m = 3$ that 
\begin{align*}
\langle g,g \rangle \ll N^{\frac{3}{2}+\eps} \quad \text{and} \quad \langle f,f \rangle \ll \sqrt{a} (\det Q)^{\frac{2}{3}}
\end{align*} 
where the latter bound is only valid for diagonal $Q$ with smallest entry $a$. For $m \ge 4$, Blomer obtains in the appendix of \cite{Sa2018} that 
\begin{align*}
\langle g,g \rangle \ll N^{m-2+\eps} + N^{\frac{m-3}{2}} \sqrt{\det Q} + (\det Q)^{1- \frac{1}{m}}. 
\end{align*}
In the special case that $\det Q$ is a fundamental discriminant, Rouse \cite{Rouse2014}  shows for quaternary forms that $\langle f,f \rangle \asymp N$ if $\min(NQ^{-1}) \ll 1$, where $\min Q $ denotes the smallest integer represented by $q(x) = \frac{1}{2}x^T Qx$. 

To state our results, we introduce further notation.  For odd $p$, the form $q(x)$ is equivalent over $\Z_p$ to 
\begin{align} \label{eq:padicform}
p^{\nu_1(p)} u_1 x_1^2+ p^{\nu_2(p)} u_2  x_2^2+ \ldots + p^{\nu_m(p)} u_m x_m^2 \quad \text{ with } u_i \in \Z_p^\times
\end{align}
and over $\Z_2$ to
\begin{align} \label{eq:2adicform}
\sum_{i=1}^{r_1} 2^{\tilde{\nu}_i} (x_{2i-1}^2 + x_{2i-1} x_{2i} + x_{2i}^2) +  \sum_{i=r_1+1}^{r_2} 2^{\tilde{\nu}_i}   x_{2i-1} x_{2i}  + \sum_{i=2r_2+1}^m 2^{\nu_i(2)} u_i x_i^2
\end{align} 
where  $u_i \in (\Z_2)^{\times}$, $0 \le 2r_1 \le 2 r_2\le m$ and $\nu_i(2) \geq 1$. We set $\nu_{2j}(2) = \nu_{2j-1}(2) := \tilde{\nu}_j$ for $j \leq r_2$ and define $
v_p(j) = \# \{1 \leq i \le m \mid \nu_i(p) \ge j\}
$. Moreover, for $s \geq 1$, we set
\begin{align} \label{eq:F(Q,s)}
F(Q,s) :=  2^{\operatorname{min}(m -2 r_2,s)} \prod_{p \mid N} p^{\mu_p(s)}  \quad \text{for} \quad  \mu_p(s) = \sum_{j=1}^{v_p(N)} \min(v_p(j),s). 
\end{align} 
Then, $F(Q,s)$ is growing in $s$ for $1 \le s \le m$ and  $N/2 \leq F(Q,s) \leq F(Q,m) = \det Q$. Furthermore, $F(Q,s)$ is a genus invariant term.  


\begin{thm} \label{thm:main} Let $Q$ correspond to a primitive, integral, positive quadratic form of level $N$ in $m$ variables and $f(z)=\theta(Q,z) - \theta(\operatorname{gen}Q,z)$. Then, 
\begin{align*}
\langle f,f \rangle  \ll  \begin{cases}[1.6]
\frac{N^{1+\eps}}{(\det Q)^{1/3}}  & \text{if } m=3, \\ 
\frac{N^{2+\eps}}{F (Q,2)} + \frac{N^{1+\eps}}{(\det Q)^{1/4}}  & \text{if } m=4, \\ 
\frac{N^{m/2+\eps}}{F (Q, \frac{m-1}{2}-\frac{1}{m})}  & \text{if } m\geq 5. 
\end{cases} 
\end{align*}
These bounds also hold for $\langle g,g \rangle$ and $g(z)= \theta(Q,z)-\theta(Q',z)$ provided that $Q'$ lies in the same genus as $Q$. 	
\end{thm}

This is a significant improvement over previous results. For example, if $m \geq 4$  and $\det Q \asymp N$, we obtain $\langle f,f \rangle \ll N^{\frac{m}{2}-1}$ which is only the square root of the bound $\langle f,f \rangle \ll N^{m-2}$ in \cite{Sa2018}. This saving increases if the determinant is large compared to the level. 

In applications, one often encounters special types of quadratic forms. A typical example are diagonal forms: 

\begin{thm} \label{thm:main2} Let $Q,N,f$ be as in the previous theorem. In addition, we assume that $Q$ is diagonal with entries $a_1 \le \ldots \le a_m$ and $m \geq 3$. Then, 
	\begin{align*} 
	\langle f,f \rangle \ll \Big(\frac{N^\frac{m}{2}}{F(Q, \frac{m}{2})} + \frac{ N }{\sqrt{a_m a_{m-1}}}\Big) N^\eps. 
	\end{align*} 
\end{thm}

Here, $F(Q,\frac{m}{2})$ is of the same size as $\det Q$ if  the common divisor of any $\lfloor \frac{m}{2} \rfloor+1$ different elements out of $\{a_1,\ldots,a_m\}$ is  bounded by a constant. 

We also provide a lower bound: 
\begin{thm} \label{thm:3} Let $Q,n,f$ be defined as in Theorem \ref{thm:main} and $M = \min_{x \in \Z^m} \frac{1}{2} x^T N Q^{-1}x$. Then, it holds for $m \geq 3$ that
	\begin{align*}
	\langle f,f \rangle \gg \frac{N^\frac{m}{2}}{\det Q} M^{1- \frac{m}{2}} + \mathcal{O}(N^\eps).
	\end{align*}
\end{thm}
This shows that, without further assumptions, our analysis is sharp. Indeed, for  $q(x) = x_1^2 + \ldots + x_{m-1}^2 + N x_m^2$ upper and lower bound in Theorem \ref{thm:main2} and \ref{thm:3} only differ by an $N^\eps$ factor. For the minimum $M$ of $NQ^{-1}$ Hermite's theorem states that $M \ll N (\det Q)^{-\frac{1}{m}}$, which yields that $\langle f,f \rangle \gg N (\det Q)^{-\frac{1}{2}-\frac{1}{m}} + \mathcal{O}(N^\eps)$. 

One of the key ingredients for the proof of Theorem \ref{thm:main} and \ref{thm:main2} is an upper bound for $r(Q,n)$. This result is of independent interest and can be found in Lemma \ref{eq:thetaseriesl}. For the genus theta series, we extend a bound of Blomer \cite{Blomer2008}  for ternary forms to  $m \geq 3$: 
\[
r(\operatorname{gen} Q,n) \ll \frac{n^{\frac{m}{2}-1}(n,N)^{\frac{1}{2}}}{\sqrt{\det Q}} (nN)^\eps.
\]
The principal use of Theorem \ref{thm:main} and \ref{thm:main2} is to improve the error  $|r(Q,n) - r(\operatorname{gen}Q,n)|$. For the corresponding results and effective lower bounds for Tartakovski's theorem with respect to $Q$, we refer the reader to  Lemma \ref{lemma:m=3}  and \ref{lemma:m=4}. 

\subsection{Applications} 
By Gauss and Siegel it is known  that every positive integer $n \not \equiv 0,4,7 \,(\operatorname{mod}8)$ can be expressed as a sum of three squares $n=x_1^2 + x_2^2 + x_3^2$ in $n^{\frac{1}{2}+o(1)}$ ways. Therefore, one would expect that we can still represent $n$ if we restrict $x_j$ to a plausible subset of the integers and there is no obvious local obstruction. A famous conjecture in this regard is that every 
\begin{align} \label{eq:n} 
n \equiv 3 \,(\operatorname{mod} 24), 5 \nmid n
\end{align} 
can be represented as the sum of three squares of primes. Current technology is not sufficient to prove this; however, using a vector sieve, Blomer and Brüdern \cite{BB2005} obtain similar findings for almost primes. The principal input for this sieve is a uniform bound for $ r(Q,n) - r(\operatorname{gen}Q,n)$, where $q(x)= \frac{1}{2}x^T Q x = l_1^2 x_1^2 + l_2^2 x_2^2 + l_3^2 x_3^2$.

By this approach, they show that every $n$ satisfying \eqref{eq:n} is represented by 
\begin{align*}
n = x_1^2 + x_2^2 + x_3^2 \text{ with } x_j \in P_{521},
\end{align*}
where $P_r$ denotes all integers with at most $r$ prime factors. Subsequently, L\"u \cite{L2007} and Cai \cite{Cai2012} optimized the sieving process by including weights and Blomer \cite{Blomer2008} improved his estimates for Fourier coefficients of the cuspidal part of the theta series. Taken together, this allows to choose $x_j  \in P_{106}$ and $x_1 x_2 x_3 \in P_{304}$. By Theorem \ref{thm:main} and a small modification of L\"u's approach, we obtain: 

\begin{kor} \label{cor:m=3} Every sufficiently large $n$ with $n \equiv 3 \,(\operatorname{mod} 24)$ and $5 \nmid n$ can be represented in the form $n=x_1^2 + x_2^2 + x_3^2$ for integers $x_1,x_2,x_3$ with $x_1 x_2 x_3 \in P_{72}$.  
\end{kor}

The dual problem is to represent an integer by squares of smooth numbers. For three variables, this is a hard task and only a small saving is possible. An application of Theorem \ref{thm:main2} yields:
\begin{kor} \label{cor:smooth} For sufficiently large $n \not\equiv 0,4,7\,(\operatorname{mod}8)$, there is a solution of ${n= x_1^2 + \ldots + x_m^2}$ with $x_1,\ldots,x_m \in \Z$ such that every prime divisor $p \mid x_1\cdots x_m$ satisfies 
\begin{align*}
&p \leq n^{\frac{\eta}{2}+\eps}, \eta = \frac{57}{58} \quad ~ \, \text{if } m=3 \text { and } n \equiv 0,4,7 \, (\operatorname{mod} 8), \\
&p \leq n^{\frac{\theta}{2}+\eps}, \theta = \frac{285}{464}  \quad  \text{if } m=4.
\end{align*}
\end{kor}

This is a slight improvement of  \cite[Theorem 2 \& 3]{BBD2009}, where $\eta = \frac{73}{74}$ and $\theta = \frac{365}{592}$. 

\textit{Notation and conventions.} We use the usual $\epsilon$-convention and all implied constants may depend on $\epsilon$.  By $[.,.], (.,.)$ we refer to the least common multiple respectively the greatest common divisor of two integers. If $p^r \mid N$, but $p^{r+1} \nmid N$, we write $p^{r} \pdiv N$.

\section{Siegel's theory and local densities}

For a holomorphic function $f$ on the upper half plane and $\gamma \in \Gamma_0(N)$ we write  
\begin{align*}
f|[\gamma]_{\frac{m}{2}}(z) :=  \begin{cases}
(cz+d)^{-\frac{m}{2}} f(\gamma z) & \text{if } m \text{ is even} \\
(\eps_d \big(\frac{c}{d}\big))^{-m}  (cz+d)^{-\frac{m}{2}} f(\gamma z) & \text{if } m \text{ is odd and } 4 \mid N ,
\end{cases} 
\end{align*} 
where $\big(\frac{c}{d}\big)$ is the extended Kronecker symbol and $\eps_d = \left(\frac{-1}{d}\right)^\frac{1}{2}$. Moreover, we denote by $S_{m/2}(N,\chi)$ the space of cusp forms of weight $\frac{m}{2} $ for $\Gamma_0(N)$ and character $\chi$. 

Two positive, quadratic forms $Q,Q'$ belong to the same class if $Q' = U^T Q U$ for $U \in \operatorname{GL}_m(\Z)$. Furthermore, they are in the same genus if they are equivalent over $\Z_p$ for all $p$, so in particular over $\Q$. There are only finitely many classes in the genus and the set 
of automorphs $o(Q) :=\{ U \in \operatorname{GL}_2(\Z) \mid U^T Q U = Q\}$ is finite. We put 
\begin{align*}
\theta(\operatorname{gen}Q,z)= \Big(\sum_{R \in \operatorname{gen}Q} \frac{1}{\#o(R)} \Big)^{-1} \sum_{R \in \operatorname{gen} Q} \frac{\theta(R,Z)}{\#o(R)} = \sum_{n} r(\operatorname{gen}Q,n) e(nz). 
\end{align*}
The level $N$ of $Q$ is defined as the smallest integer such that $NQ^{-1}$ is integral with even diagonal entries. Throughout this work, we assume that $q(x) = \frac{1}{2} x^T Qx$ is primitive. This implies that $NQ^{-1}$ has level $N$ and that 
$
\frac{1}{2} N \mid \det Q \mid 2 N^{m-1}. 
$

Recall that $\theta(Q,z) -  \theta(\operatorname{gen}Q,z) \in S_{m/2}(N,\chi)$. The space $S_{3/2}(N,\chi)$ contains a subspace $U$ generated by theta functions of the form $\sum \psi(n) n e(tn^2z)$ for some real character $\psi$ and $t \mid 4N$. Their non-vanishing Fourier coefficients are of size $\asymp n^\frac{1}{2}$ which is the order of magnitude of $r(\operatorname{gen}Q,z)$, cf.\ \eqref{eq:rgen}. Hence, for $m=3$ the main term needs to be modified.


Two quadratic forms $q_1,q_2$ with matrices $Q_1 = S^T Q_2 S$ and $S \in \operatorname{GL}_m(\Z)$ in the same genus belong to the same spinor genus if  $S \in O_{Q}(A_2) \bigcap_{p} O'_{Q_p}(A_2)GL_{m}(\Z_p)$, where $O'_{Q_p}(A)$ is the subgroup of $p$-adic automorphs $O_{Q_p}(A)$ of determinant and spinor norm 1, cf. \cite[Section 55]{Me}. We set 
\begin{align*} 
\theta(\operatorname{spn}Q,z)= \Big(\sum_{R \in \operatorname{spn}Q} \frac{1}{\#o(R)} \Big)^{-1} \sum_{R \in \operatorname{spn} Q} \frac{\theta(R,Z)}{\#o(R)}.
\end{align*}
Then, $\theta(Q,z) - \theta(\operatorname{spn}Q,z) \in U^{\bot} \subseteq S_{3/2}(N,\chi)$. For $g(z) = \sum_n b(n) e(nz) \in U^\bot $, we have by \cite[Theorem 1]{Wa2017} for 
$n=tv^2w^2$ with squarefree $t$ and $(w,N)=1$ that 
\begin{align}  \label{eq:2}
a(n) \ll \norm{g} n^{\frac{1}{4}} \Kl{1 + \frac{n^{\frac{3} {14}}}{N^{\frac 1 7}} + \frac{n^{\frac{3} {16}}}{N^{\frac{1} {16}}}+  \frac{\sqrt{v (n,N)}}{\sqrt{N}}} (nN)^{\eps}.
\end{align}
This saving in $n$ is sufficient to obtain an asymptotic formula for $r(Q,n)$, since by a local argument we can assume that $(n,N)$ is small. 

In principle, the Fourier coefficients $ r (\operatorname{spn}Q,n)$ of $\theta(\operatorname{spn}Q,z)$ can be computed locally in a similar fashion as $r(\operatorname{gen}Q,n)$. However, this process is complicated and tedious. Luckily, in many situations  $r (\operatorname{spn}Q,n)$ and $r(Q,n)$ coincide. For example, this holds if 
\begin{align} \label{eq:gen=spn}
\begin{split}
\large{\text{\textbullet}}~ &n \notin \{tm^2 \mid  4t \mid N, m \in \N\} \text{ or} \\[-3pt] 
\large{\text{\textbullet}}~ &q \simeq u_1 p^{\nu_1} x_1^2 + u_2 p^{\nu_2} x_2^2 + u_3 p^{\nu_3} x_2^2 \text{ over } \Z_p, \text{at least two } \nu_i  \text { are equal for}  \\[-3pt] &\text{odd } p \text{ and all three } \nu_i \text{ are equal for } p=2. 
\end{split}
\end{align}
To obtain upper and lower bounds for $r(\operatorname{gen}Q.n)$, we need to evaluate the $p$-adic densities given in \eqref{eq:padicdensities}. If $p \nmid 2nN$, an easy computation shows that
\begin{align} \label{eq:siegel1}
\beta_p(n,Q) = \begin{cases}
1- p^{-\frac{m}{2}}\left(\frac{(-1)^\frac{m}{2} \det Q }{p^\frac{m}{2}}\right) & \text{if } m \text{ even}, \\
1 + p^{\frac{1-m}{2}} \left(\frac{(-1)^\frac{m-1}{2} n \det Q }{p^\frac{m}{2}}\right) & \text{if } m \text{ odd},
\end{cases}
\end{align}
cf.\ \cite[Hilfssatz 12]{Si1935}. For $p \nmid 2N$, we have  by  \cite[Hilfssatz 16]{Si1935} that 
\begin{align} \label{eq:siegel2}
1- p^{-r} \leq \beta_p(n,Q) \leq 1 +  p^{r} 
\end{align}
where $r=  \frac{m}{2}$ for even $m$ and $r= \frac{1-m}{2}$ for odd $m$. For the remaining $p$-adic densities, we apply a formula of Yang \cite{Ya1998}. For odd $p$ and  $\nu_i, u_i$  defined as in \eqref{eq:padicform}, we set
\begin{align*}
V(l) = \{1 \leq i \leq  m \mid  \nu_i - l <0 \text{ is odd} \}
\end{align*}
and  
\begin{align} \label{def:d(l)v(l)}
d(l) = l +  \frac{1}{2}\sum_{\nu_i<l} (\nu_i-l), \quad v(l) = \left(\frac{-1}{p}\right)^{\lfloor \frac{\# V(l)}{2} \rfloor}\prod_{i \in V(l)} \left(\frac{u_i}{p}\right).
\end{align}
\begin{lemma} \cite[Theorem 3.1]{Ya1998}  \label{thm:Yang} For odd $p$ and $n= p^a t$ with $(t,p)=1$ we have that 
	\begin{align*}
	\beta_p(Q,n) = 1 + (1+p^{-1}) \sum_{\substack{0 < l \leq a \\ \#V(l)\, \text{even}}} v(l) p^{d(l)} + v(a+1) p^{d(a+1)} f(n)
	\end{align*}
 where
 \begin{align*}
 f(n) = \begin{cases}
 - \frac{1}{p} & \text{if } \#V(a+1) \text{ is even,} \\
 \left(\frac{t}{p}\right)\frac{1}{\sqrt{p}} & \text{if } \#V(a+1) \text{ is odd}.
 \end{cases}.
 \end{align*}
\end{lemma}

To compute the densities for odd $p \mid N, p\nmid n$, we make a case distinction according to $V(1) =  \#\{i \mid \nu_i  =0 \}$. If $V(2) \geq 2$ it follows by the theorem above that 
\begin{align} \label{cor:Yang}
1 - \frac{1}{p} \leq \beta_p(Q,n) \leq 1 + \frac{1}{p}.
\end{align}
For $V(1) =1$, we obtain $\beta_p(Q,n) = 1 + \Big( \frac{n}{p} \Big) \prod_{i \in V(1)} \Big(\frac{u_i}{p} \Big)$. 

For the computation of the 2-adic densities Yang provides a slightly more complex formula \cite[Theorem 4.1]{Ya1998}. As a consequence, we obtain the following upper bound: 

\begin{kor}\label{cor:theta} It holds that 
	\begin{align*}
	\prod_p \beta_p(Q,n) \ll (n,N)^{\frac{1}{2}} (nN)^\eps.
	\end{align*}
\end{kor}
\begin{proof}
By \eqref{eq:siegel1}, \eqref{eq:siegel2} and \eqref{cor:Yang} we already know that 
\begin{align*}
\prod_p \beta_p(Q,n) \ll  (nN)^\eps \prod_{p \mid 2 (n,N)} \beta_p(Q,n).
\end{align*} 
Let $(n,N) = p^b$. Then, $d(l) \leq \frac{b}{2}$ for all $l$. Hence, Lemma \ref{thm:Yang} gives for odd $p$ that $\beta_p(Q,n) \ll p^{\frac{b}{2}}$ as $\abs{v(l)}=1$. For $p=2$ we apply \cite[Theorem 4.1]{Ya1998}. The term $d(k)$, defined in \cite[(4.3)]{Ya1998},  satisfies $d(k) \leq  2^{\frac{b+1}{2}}$ which implies that \cite[(4.4)]{Ya1998} is bounded by $\ll 2^{\frac{b}{2}}$ and hence, also $\beta_2(n,Q)$.
\end{proof}

As a consequence, we obtain by \eqref{eq:rgen} that
\begin{align} \label{eq:genusupperbound}
r(\operatorname{gen} Q,n) \ll \frac{n^{\frac{m}{2}-1}(n,N)^{\frac{1}{2}}}{\sqrt{\det Q}} (nN)^\eps. 
\end{align}

To obtain lower bounds for $\beta_p(Q,n)$ we follow the approach of Hanke \cite{Ha2004}. For $q \simeq  u_1 p^{\nu_1} x_1^2 + \ldots + u_m p^{\nu_m} x_m^2 $ in the form \eqref{eq:padicform} and a primitive solution  $x \in (\Z_p)^m$ of $q(x) = n$, we set 
\begin{align*}
\nu(x) = \min \{ \nu_i \mid x_i \in \Z_p^\times\}.
\end{align*}
We say that $x$ is of \textit{good} type if $\nu(x)=0$, of bad type I if $\nu(x)=1$ and of bad type~II if $\nu(x) \ge 2$. If there is a solution of \text{good} type, we count solutions modulo $p$ and apply Hensel's lemma to compute $\beta_p(q,n)$. As a result, we obtain $\beta_p(q,n) \geq 1- \frac{1}{p}$ and $\beta_2(q,n) \geq \frac{1}{4}$. 

In all other cases, we apply reduction maps. Therefore, let $q \equiv q_0 + p q_1 + p^2 q_2$, where $q_0$ consist of all diagonal terms with $\nu_i=0$ and $q_1$ of those with $\nu_i=1$. If there are no good solutions, but at least one of bad type I, we reduce modulo $p$ which gives $\beta_p(q,n) \geq p^{1-\dim q_0} \beta_p(q',\frac{n}{p})$ for $q' = p q_0  + q_1 + p q_2$. If there are only bad type~II solutions, we reduce modulo $p^2$ which gives $\beta_p(q,n) \geq p^{2-\dim q_0-\dim q_1} \beta_p(q'',\frac{n}{p^2})$ for $q'' = q_0 + p q_1 + q_2$. As a consequence, we obtain:

\begin{lemma} \label{lemma:lowerbound} Assume there is a primitive solution $q(x) = n$ over $\Z_p$ with $\nu := \nu(x)$ as above and $m \geq 4$. Then, 
\begin{align*}
r_p(q,n) \geq 1-\frac{1}{p} \quad \text{if } p \equiv 1 \,(\operatorname{mod} 4).
\end{align*}
Furthermore, we have for $m=4$ that 
\begin{align*}
r_p(q,n) \geq \frac{1- \frac{1}{p}}{\sqrt{(p^{\nu},n)}} \text{ for } p \equiv 3\,( \operatorname{mod} 4) \text{ and }  r_2(q,n) \geq  \frac{1}{32 \cdot \sqrt{(2^{\nu},n)}}
\end{align*}	
and for $m \geq 5$ that 
\begin{align*}
r_p(q,n) \geq \frac{1- \frac{1}{p}}{(p^{\nu},n)} \text{ for } p \equiv 3 \,(\operatorname{mod} 4) \text{ and } r_2(q,n) \geq  \frac{1}{32 \cdot (2^{\nu},n)}.
\end{align*}
\end{lemma}
\noindent
\textbf{Remark}. It seems counterintuitive that the bounds for $m=4$ are stronger. One reason is  the assumption that $n$ is primitively locally represented which eliminates the worst case scenario for the lower bound, i.e.\ that $q \sim  x_1^2 + x_2^2 + p x_3^2 + p x_4^2$ and $p \equiv 3 \,(\operatorname{mod}4)$. 
\begin{proof} The case $m=4$ is essentially proven in \cite[Lemma 2]{Ro2019}. If $p \equiv 1\, (\operatorname{mod} 4)$ and $\dim q_0 \geq 2$,  there is always a good solution. Moreover, a bad type solution cannot exist for $q \simeq u_1 x_1^2 + u_2 x_2^2 + u_3 p x_3^2 + u_4 p x_4$ if $p \equiv 3 \,(\operatorname{mod}4)$.  Indeed, all solutions of $x_1^2 + x_2^2 + p x_3^2 + p x_4^2 \equiv 0 \,(\operatorname{mod} p^2)$ satisfy $x \equiv 0 \,(\operatorname{mod}p)$ for $p \equiv 3 \,(\operatorname{mod}4)$. The cases $m \geq 5$ work analogously to $m=4$.  
\end{proof}

Since $\nu \leq v_p(N)$, it follows for  $m=4$  that $\prod_p \beta_p(q,n) \gg (n,N)^{-\frac{1}{2}} (nN)^{-\eps}$ if $n$ is primitively represented. A form in 5 variables is isotropic over $\Z_p$ and hence if we sort $\nu_i$ by value, there is always a primitive solution $x$ such that $\nu(x) \le \nu_5$. Hence, we obtain for $m \geq 5$ that
\begin{align*}
\prod_p \beta_p(q,n) \gg  \prod_{p \mid N} \big( p^{\frac{v_p(\det Q)}{m-4}}, p^{v_p(N)},n\big) ^{-1}(nN)^{-\eps}.  
\end{align*}
For a given form, it is often possible to find a primitive solution with $\nu=0$ for all $p \mid (n,N)$ which yields $\prod_p \beta_p(q,n) \gg  (nN)^{-\eps}$. 



	


\section{A uniform bound for the inner product}

The aim of this section is to prove Theorem \ref{thm:main}, \ref{thm:main2} and \ref{thm:3}. Recall that the Siegel domain $\mathcal{S}(\delta,\eta)$ consists of all positive matrices $M$ that allow for a decomposition into 
\begin{align*}
M = V^T D V, 
\end{align*}
where $V$ is an upper triangular matrix with ones on the diagonal and all other entries $\abs{v_{i,j}} \leq \eta$ and $D$ is a diagonal matrix with entries  $a_1, \ldots, a_{m} $ that satisfy $0 < a_i \leq \delta a_{i+1}$. 

Let $Q$ denote a positive, integral matrix with even diagonal entries. By \cite[p.\,259]{Ca87}, the class of $Q$ contains an element $Q' \in \mathcal{S}(\frac{4}{3},\frac{1}{2})$. We decompose  $Q'=V^T DV$ as above and call $a_1,\ldots,a_m$ the \textit{associated diagonal entries} of $Q$. Since $a_1$ equals the upper left entry of $Q'$ it is integral and thus, $a_1 \geq 1$. This construction of $a_1,\ldots,a_m$ with respect to $Q$ is not unique. If we assume that $Q'$ is Hermite-reduced, then $a_i$ is of the size as the $i$-th successive minimum of $Q$, cf.\ \cite[\S12,\,Theorem 3.1]{Ca87}.

\begin{lemma} \label{eq:thetaseriesl} Let $Q$ correspond to a positive, integral quadratic form of level $N$ with associated diagonal entries $a_1,\ldots,a_m$. Then,  
\[a_1 \cdots a_m = \det Q \quad \text{and} \quad \Big(\frac{3}{4}\Big)^{i-1} \le a_i \leq \Big(\frac{4}{3}\Big)^{m-i} N \text{ for every } 1 \le i \le m.\] Furthermore, it holds that 
\begin{align*} 
\sum_{x \leq l} r(Q,x)^2 \ll \Big(1 + \frac{l^\frac{1}{2}}{\sqrt{a_1}} + \frac{l}{\sqrt{a_1 a_2}} + \ldots  + \frac{l^{m-1}}{\sqrt{a_1 a_2} a_3 \cdots a_m} \Big) l^\eps.
\end{align*}  
More precisely, for $2 \le j \le 2m-2$ the  $j$-th term is given by $\frac{l^{j/2}}{ A_j}$ where
\begin{align*}
A_j = \prod_{i=1}^{\lfloor \frac{j+3}{2} \rfloor} \sqrt{a_i} \prod_{k=3}^{\lfloor \frac{j+2}{2} \rfloor} \sqrt{a_k} = \det Q \Big(\sqrt{a_1 a_2} \prod_{i= \lfloor \frac{j+5}{2} \rfloor}^{m} \sqrt{a_i} \prod_{k=\lfloor \frac{j+4}{2} \rfloor}^{m} \sqrt{a_k}\Big)^{-1}.
\end{align*}
Moreover, we have 
\begin{align*} 
\sum_{x \leq l} r(NQ^{-1},x)^2 \ll  \Big(1 + \frac{ l^{\frac{1}{2}}\sqrt{a_{m}}}{\sqrt{N}} + \sum_{j=2}^{2m-2}  \frac{l^\frac{j}{2} \tilde{A}_j }{N^\frac{j}{2} } \Big) l^\eps, 
\end{align*}
where
\begin{align*}
\tilde{A}_j \! = \! \! \prod_{i=m-\lfloor \frac{j+1}{2} \rfloor}^{m} \! \sqrt{a_i} \prod_{k=m-\lfloor \frac{j}{2} \rfloor}^{m-2} \sqrt{a_k}.
\end{align*}
\end{lemma}
\begin{proof}
Since $r(Q,n) = r(U^T Q U, n)$ for any $U \in \operatorname{GL}_{m}(\Z)$, we may assume that $Q \in \mathcal{S}(\frac{4}{3},\frac{1}{2})$. If we decompose $Q=V^T D V$ with $V$ and $D$ as above, we obtain that
\begin{align*}
q(x) = \frac{1}{2} x^T Q x = \frac{a_1}{2} (x_1+ v_{1,2} x_2 + \ldots + v_{1,m} x_{m})^2 + \dots + \frac{a_{m}}{2} x_{m}^2.
\end{align*}
We count the number of solutions of $q(x) =n$. For $x_3,\ldots,x_{m}$ there are at most 
\begin{align*}
\left(2 \sqrt{\frac{n}{a_3}}+1\right) \cdots \left(2 \sqrt{\frac{n}{a_{m}}}+1\right)
\end{align*}
choices. 

After that, we are left with a binary quadratic form that has $\mathcal{O}(n^{\eps})$ solutions uniformly in all parameters. This yields 
\begin{align*} 
r(Q,n) \ll_m \left(1 + \frac{n^{1/2}}{\sqrt{a_3}}+ \ldots + \frac{n^{\frac{m}{2}-1}}{\sqrt{a_3\cdots a_{m}}} \right) n^{\epsilon}.
\end{align*}
If we want to count all solutions of $q(x) \leq l$, we proceed as before, but bound the number of solutions for $x_1$ and $x_2$ by  
$
\big(2\sqrt{\frac{l}{a_1}} +1 \big) \text{ and } \big(2\sqrt{\frac{l}{a_2}}+1\big). 
$
Thus, we obtain 
\begin{align*}
\sum_{x \leq l} r(Q,x) \ll 1 + \max_{1 \leq j \leq m} \frac{l^{\frac{j}{2}}}{ \sqrt{a_1} \ldots \sqrt{a_j}}.
\end{align*}
The first claim now follows from these two bounds and $a_1 \cdots a_m = \det Q$. 

Next, we consider the quadratic form corresponding to $NQ^{-1}$. Note that
\begin{align*}
N Q^{-1} = N (V^T D V)^{-1} = V^{-1} (N D^{-1}) V^{-T}. 
\end{align*}
The diagonal entries of $ND^{-1}$ are 
\begin{align*}
\frac{N}{a_1}, \ldots, \frac{N}{a_{m}} \quad \text{with} \quad \frac{N}{a_i} \ge \frac{3}{4}  \frac{N}{a_{i+1}}. 
\end{align*}
The lowest diagonal entry $\frac{N}{a_{m}}$ of $D$ equals the lower right entries of $N Q^{-1}$ since $V^{-1}$ is an upper and $V^{-T}$ a lower diagonal matrix with ones on the diagonal. Hence, $N/a_{m} \in \Z$ and thus $a_m \leq N$. Together with $a_i \geq 1$ and $a_i \leq 4 a_{i+1} / 3$, this gives the bounds for $a_i$ claimed in the lemma. 

Moreover, it holds that 
\begin{align*}
\tilde{q}(x) &:= \frac{1}{2} x^T N Q^{-1} x =\frac{1}{2} x^T  V^{-1} (N D^{-1}) V^{-T}  x \\ &= \frac{N}{2 a_1} x_1^2 + \ldots + \frac{N}{2 a_m} ( w_{1,m} x_1 + \ldots + w_{m-1,m} x_{m-1}^2 + x_m)^2,
\end{align*}
where $w_{ij}$ are the entries of $V^{-1}$. Their absolute value $|w_{ij}|$ is bounded from above by a constant. To see this, write  $V= I + M$ for the $m\times m$ identity $I$ and a nilpotent matrix $M$. Then, a simple calculation shows that $V^{-1} = (I+M)^{-1} = I + \sum_{k=1}^{m-1} (-1)^k M^k$. 

Next, we count solutions of $\tilde{q}(x) =n$. For $x_1,\ldots x_{m-2}$ there are \[\Big(2 \sqrt{n\cdot \frac{ a_1}{N}}+1\Big) \cdots \Big(2 \sqrt{n \cdot \frac{ a_{m-2}}{N}}+1 \Big)\] choices and $\mathcal{O}(n^\eps)$ for $x_{m-1} $ and $x_m$. It follows that
\begin{align} \label{eq:qf}
r(NQ^{-1},n) \ll \left(1 + \frac{n^{1/2} \sqrt{a_{m-2}}}{\sqrt{N}}+ \ldots + \frac{n^{\frac{m}{2}-1}\sqrt{a_1 \cdots a_{m-2} }}{N^{\frac{m}{2}-1}} \right) n^{\epsilon}.
\end{align}
Moreover, we obtain by the same argument as above  that
\begin{align*}
\sum_{x \leq l} r(NQ^{-1},x) \ll 1 + \max_{1 \leq j \leq m} l^{\frac{j}{2}} \cdot \frac{ \sqrt{a_{m-j+1}} \ldots \sqrt{a_m}}{N^{\frac{j}{2}}}.
\end{align*}
Combining these bounds yields the second claim. 
\end{proof}

The level of a form is easily determinable by the local splittings: 

\begin{lemma} \label{lemma:help} Let $q= \frac{1}{2}x^T Q x$ be equivalent over $\Z_p$ to \eqref{eq:padicform} for odd $p$ and to \eqref{eq:2adicform} for $p=2$. Set 
\begin{align*}
\nu(p) = \max_{1 \leq i \leq m} \nu_i(p) \text{ for odd } p  \text{ and }   \nu(2) = \max_{1 \leq i \leq m} ( \nu_i(2) + 2 \delta_{ i  >  2 r_2}).
\end{align*}
Then, the level $N$ of $Q$ is given by $N = \prod_{p} p^{\nu(p)}$. 
\end{lemma}
\begin{proof}
For odd $p$, choose $U \in \operatorname{GL}_m(\Z_p)$ such that $\tilde{Q} = U^T Q U$ is diagonal with entries $u_i p^{\nu_i(p)}$. It follows that $N \tilde{Q}^{-1} = N U^{-1} Q^{-1} U^{-T}$, thus $p^{\nu_i(p)} \mid N$ for all $p$. For the other inclusion, we ise that every entry of $p^{\nu(p)} Q^{-1} = U^{-1} p^{\nu(p)} \tilde{Q}^{-1} U^{-T}$ is in $\Z_p \cap \Q = \{\frac{a}{b} \in \Q \mid p \nmid b \}$. For $p=2$, we choose $U \in \operatorname{GL}_m(\Z_2)$ such that $\tilde{Q} = U^T Q U$ is of the form \eqref{eq:2adicform}. By 
\begin{align*}
\begin{pmatrix}
2 & 1 \\ 1 & 2
\end{pmatrix}^{-1} = \frac{1}{3} \begin{pmatrix}
2 & -1 \\ -1 & 2
\end{pmatrix}, \quad \begin{pmatrix}
0 & 1 \\ 1 & 0
\end{pmatrix}^{-1} = \begin{pmatrix}
0 & 1 \\ 1 & 0
\end{pmatrix},
\end{align*}
and the same argument as before, the claim follows
\end{proof}

To evaluate the Fourier coefficients of $\theta(Q,z)|\left(\begin{smallmatrix}
a & b \\ c & d
\end{smallmatrix}\right)$, we construct a diagonal matrix $D$ that depends on $Q$ and $c$. To this end, let $c= \tilde{c} 2^t$ for odd $\tilde{c}$  and fix $U \in \operatorname{SL_m(\Z)}$ such that $\frac{1}{2} x^T U^T Q U x \equiv  \sum_{i=1}^m q_i x_i^2 \,(\operatorname{mod} \tilde{c})$ and 
\begin{align*}
\frac{1}{2}x^T U^T Q U x \equiv \sum_{i=1}^{r} 2^{\tilde{\nu}_i} g_i(x_{2i-1}, x_{2i})
+ \sum_{i=2r+1}^m 2^{\nu_i} u_i x_i^2 ~(\operatorname{mod} 2^t)
\end{align*}
for odd $u_i \in \Z$ and $g_i(x_{2i-1}, x_{2i}) = (x_{2i-1}^2 + x_{2i-1} x_{2i} + x_{2i}^2)$ or $x_{2i-1} x_{2i}$, $\tilde{\nu}_i,\nu_i \le t$. We set $\tilde{d}_i = (q_i,\tilde{c},N)$, $t_{2i-1} = t_{2i} = \tilde{\nu}_i$ for $i \leq r$, 
\begin{align*}
t_i = \begin{cases} \nu_i & t \leq \nu_i +1, \\
\nu_i +1 & t \geq \nu_i +2.
\end{cases} \text{ for } i \geq 2r  \text{ and }  D = \operatorname{diag}(\tilde{d}_1 2^{t_1} ,\ldots,\tilde{d}_m 2^{t_m} ).
\end{align*} 
Furthermore, we set $\hat{d} := d/2$ if $2^2 \pdiv N, 2 \pdiv d$ and $2r<m$, and $\hat{d} := d$ otherwise. For future purposes, we also define 
\begin{align} \label{def:eta}
\eta(x_i) = \begin{cases} 1 & \text{if }  t \neq \nu_i+1, \\
\sqrt{2}  &\text{if }  t = \nu_i +1, x_i  \text{ odd}, \\ 0 & \text{if }
t = \nu_i +1, x_i  \text{ even} \end{cases}
\end{align}

\begin{lemma} \label{lemma:8}
Let $U,D$ as above and set $\tilde{Q} = U^T Q U$. Then $\frac{N}{\hat{d}} D  \tilde{Q}^{-1} D$ corresponds to a positive, integral quadratic form of level $\le N$. 
\end{lemma}
\begin{proof}
By construction $2 D^{-1} \tilde{Q}$ corresponds to a positive, integral quadratic form. Over $\Z_p$ and $p$ odd, $q= x^T D^{-1} \tilde{Q} x$ is equivalent  to 
$
\tilde{u}_1 \frac{p^{\nu_1}}{(p^{\nu_1},d)} x_1^2 + \ldots + \tilde{u}_m \frac{p^{\nu_m}}{(p^{\nu_m},d)} x_m^2.
$
Over $\Z_2$, $q$ is equivalent to
\begin{align*}
\sum_{i=1}^{r_2} 2^{\nu_{i}+1 -t_i} g(x_{2i-1},x_{2i})+ \sum_{i=2r_2+1}^m 2^{\nu_{i}+1 -t_i}u_i x_i^2,
\end{align*} 
By Lemma \ref{lemma:help}, this implies that $2 D^{-1} \tilde{Q}$ has level $4 \frac{N}{d}$ if  the maximum of $2^{\nu_{i}+1 -t_i} + 2 \delta_{i> 2r}$ is  attained for $i>2r$ and $t = \nu_i+1$ and $2 \frac{N}{d}$ otherwise. Hence, $\frac{N}{d} D \tilde{Q}^{-1} D$ has even diagonal entries unless $2^2 \pdiv N, 2 \pdiv d$  and $2r < m$. 
\end{proof}

For simplicity, we omit the weight $\frac{m}{2}$ in the slash operator $|[\varrho]_{\frac{m}{2}}$.  

\begin{lemma} \label{lemma:trafothetaseries} Let $\varrho = \left(\begin{smallmatrix} a & b \\ c & d \end{smallmatrix}\right) \in \operatorname{SL}_2(\Z)$, $\varrho \notin \Gamma_0(N)$, $D$ and $\tilde{Q}$ as in Lemma \ref{lemma:8}. Then,
\begin{align*} 
\theta(Q,z)|[\varrho] =	  \frac{\sqrt{\det D }}{\sqrt{\det Q}}   \sum_{x \in \Z^{m}} \eta(x) \omega e\Big(\frac{1}{2} x^T D \tilde{Q}^{-1} D x z\Big),
\end{align*}
for some $|\omega|=1$ and $\eta(x) := \prod_{i=1}^m \eta_i(x_i)$ with $\eta_i$ defined as in \eqref{def:eta}. For $\hat{d}$ as above and $S := \frac{N}{\hat{d}}D \tilde{Q}^{-1} D$, it follows that 
\begin{align*}
\theta(Q,z)|[\varrho]  = \sum_{n \geq 0} a(n) e\Big(\frac{\hat{d} nz}{N}\Big) \text{ with } |a(n)|\leq 2^{\frac{m}{2}} \frac{\sqrt{\det D }}{\sqrt{\det Q}}  r(S,n).
\end{align*}
If $d=(c,N)=1$, we have that $|a(n)| = \frac{r(NQ^{-1})}{\sqrt{\det Q}}$. 
\end{lemma}

\begin{proof}
	By the factorization $	\varrho =  \left(\begin{smallmatrix}
	1 & a / c \\ & 1 
	\end{smallmatrix}\right) \left(\begin{smallmatrix}
	& - 1 / c \\ c
	\end{smallmatrix}\right)\left( \begin{smallmatrix}
	1 & d / c \\ & 1 
	\end{smallmatrix}\right)$ and the transformation formula for the generalized theta series, cf.\ \cite[p.\,575]{Si1935}, we get that
	\begin{align} \label{eq:transform}
	\theta(Q,z) | [\varrho] = \sum_{x \in \Z^{m}} \alpha(x,Q,\varrho) e\Big(\frac{1}{2}  x^T Q^{-1} x z\Big)
	\end{align}
	where
	\begin{align} \label{eq:alpha}
	\alpha(x,Q,\varrho) = \frac{e(\frac{3m}{8})}{c^{\frac{m}{2}} \sqrt{\det Q}} e\Big(\frac{d}{2 c} x^T Q^{-1} x\Big) \sum_{v \in (\Z/c\Z)^{m}} e \Big(\frac{1}{c} v^T x + \frac{a}{2 c}v^T Q v\Big).
	\end{align}
We write $c = \tilde{c} \, 2^t$ with odd $\tilde{c}$ and decompose this sum into
 \begin{align*}
 \sum_{v \in (\Z/\tilde{c}\Z)^{m}} e \Big(\frac{a 2^{t} \frac{1}{2} v^T Q v + v^Tx}{\tilde{c}} \Big) \sum_{w \in (\Z/2^t\Z)^{m}}  e \Big(\frac{a \tilde{c} \frac{1}{2} w^T Q w + w^Tx}{2^t} \Big).
 \end{align*}
Next, we choose $U \in \operatorname{SL}_m(\Z)$ as in Lemma \ref{lemma:8} and substitute  $v,w$ by $Uv, Uw$. Since
$ \frac{1}{2}v^TU^TQUv \equiv \sum_{i=1}^m q_i v_i^2 \,(\operatorname{mod}\tilde{c})$ we obtain for $G(a,b,c) := \sum_{x\operatorname{mod} c} e\Big(\frac{ax^2+bx}{c}\Big)$ that
	\begin{align*}
 \sum_{v \in (\Z/\tilde{c}\Z)^{m}} e \Big(\frac{a 2^{t} \frac{1}{2} v^T Q v + v^Tx}{\tilde{c}} \Big) = 	\prod_{i=1}^m G(q_i a 2^t,u_i^Tx,\tilde{c}) 
\end{align*}
where $u_i$ denotes the $i$-th column of $U$. The $i$-th sum vanishes unless $d_i \mid u_i^T x$ for $d_i := (\tilde{c},q_i)$. By completing the square, the previous display equals 
	\begin{align} \label{eq:gaussum}
	\tilde{c}^\frac{m}{2} \eps_{\tilde{c}}^m \prod_{i=1}^m \sqrt{\tilde{d}_i}    \left(\frac{a 2^t \tilde{q}_i }{\tilde{c}_i}\right) e\Big(-\frac{\overline{ 2^{t+2} a \tilde{q}_i} \tilde{x}_i^2}{\tilde{c}_i}\Big)
	\end{align} 
	where $\tilde{q}_i = q_i / d_i, \tilde{c}_i = \tilde{c}/d_i$ and $\tilde{x}_i=u_i^T x / d_i$.	The absolute value of the term above is $\tilde{c}^\frac{m}{2} \sqrt{\tilde{d}_1} \cdots \sqrt{\tilde{d}_m}$. Furthermore, we have for $U$ as in Lemma \ref{lemma:8} that 
	\begin{align*}
	\frac{1}{2}x^TU^TQUx \equiv \sum_{i=1}^{r_1} 2^{\tilde{\nu}_i} g(x_{2i-1},x_{2i}) + \sum_{i=r_1+1}^{r_2} 2^{\tilde{\nu}_i}   x_{2i-1} x_{2i}  + \sum_{i=2r_2+1}^m 2^{\nu_i} b_i x_i^2 \,(\operatorname{mod}2^t),
	\end{align*}
	where $b_i$ are odd integers, $0\leq 2r_1 \leq 2r_2 \leq m$, $g(x_{2i-1},x_{2i}) = x_{2i-1}^2 + x_{2i-1}x_{2i} + x_{2i}^2$ and $\tilde{\nu_i},\nu_i \leq t$. For $i > 2r_2$, we need to evaluate $	G(2^{\nu_i} b_i a  \tilde{c}, u_i^T x, 2^t )$. These Gauss sums vanish unless $u_i^T x \equiv 0\,( \operatorname{mod} 2^{\nu_i})$. For $\tilde{x}_i =\frac{u_i^T x}{2^{\nu_i}}$, we obtain   
	\begin{align*}
	G(2^{\nu_i} b_i a  \tilde{c}, u_i^T x, 2^t ) = \begin{cases} 2^{\nu_i} & t= \nu_i, \\
	2^{\nu_i+1} \delta_{\tilde{x}_i \equiv 1 \operatorname{mod} 2} & t = \nu_i+1, \\
	2^{\nu_i+1+t/2}  \omega  e\Big(\frac{-\overline{b_i a \tilde{c}} (\tilde{x}_i/2)^2}{2^{t-\nu_i}}\Big) \delta_{\tilde{x}_i \equiv 0 \operatorname{mod} 2} & t \geq \nu_i +2
	\end{cases}
	\end{align*}
	where $\omega := \eps_{ b_i a  \tilde{c}}^{-1} \left(\frac{2^{t-\nu_i}}{b_i a  \tilde{c}}\right) e(1/8)$. For forms of type $2^{\tilde{\nu}} (x_1^2 + x_1 x_2 +
x_2^2)$, we have that 	
 \begin{align}  \label{eq:firstbinary}
\sum_{w_1 (\operatorname{mod} 2^{t})} e\Big( \frac{ w_1^2 a \tilde{c} 2^{\tilde{\nu}}  + w_1 u_1^T x_1}{2^t} \Big) &\sum_{w_2 (\operatorname{mod} 2^{t})}  e\Big( \frac{  w_2^2 a \tilde{c} 2^{\tilde{\nu}} + w_2 (u_1^T x_1 + a \tilde{c} 2^{\tilde{\nu}} w_1)}{2^t} \Big) \\ \notag
&=  2^{\tilde{\nu}+t} \zeta e\Big(\frac{\overline{3a \tilde{c}} (\tilde{x}_1^2 - \tilde{x}_1 \tilde{x}_2 +\tilde{x}_2^2  )}{2^{t-\tilde{\nu}}}\Big) \delta_{ \tilde{x}_1 \in \Z, \tilde{x}_2 \in \Z}
\end{align} 
where $\tilde{x}_i := \frac{u_i^T x}{2^{\tilde{\nu}}}$ and $\zeta$  is a root of unity that does not depend on $x$.

For $t= \tilde{\nu}$ the claim is obvious. If $t = \tilde{\nu}+1$, the latter sum vanishes unless $\tilde{x}_2 + a \tilde{c} w_1$ is odd and  \eqref{eq:firstbinary} equals 
\begin{align*}
2^{\tilde{\nu}+t} e\Big(\frac{\overline{a \tilde{c}} (\tilde{x}_1^2 - \tilde{x}_1 \tilde{x}_2 + \tilde{x}_2^2)}{2}\Big). 
\end{align*}
If $t \geq  \tilde{\nu}+2$, the second sum vanishes unless $\tilde{x}_1 + a \tilde{c} w_1$ is even. In this case \eqref{eq:firstbinary} equals
\begin{align*}
2^{\frac{3\tilde{\nu}+t+1}{2}} &\omega \sum_{w_1 (\operatorname{mod} 2^{t-\tilde{\nu}-1})} e\Big(\frac{(2 w_1 - \overline{a \tilde{c}} \tilde{x}_2)^2 a \tilde{c} + (2 w_1 - \overline{a \tilde{c}} ) \tilde{x}_1 - a v_1^2}{2^{t-\tilde{\nu}}} \Big) \\ 
&=  2^{\frac{3\tilde{\nu}+t-1}{2}} \omega  e\Big(\frac{\overline{a \tilde{c}} (-\tilde{x}_2^2+ \tilde{x}_1 \tilde{x}_2)}{2^{t-\tilde{\nu}}}\Big) G(3a\tilde{c},2(2\tilde{x}_2-\tilde{x}_1),2^{t-\nu}) \\
&= 2^{\tilde{\nu}+t} \zeta e\Big(\frac{\overline{3a \tilde{c}} (\tilde{x}_1^2 - \tilde{x}_1 \tilde{x}_2 +\tilde{x}_2^2  )}{2^{t-\tilde{\nu}}}\Big).
\end{align*} 

For forms of type $2^{\tilde{\nu}} x_{1} x_{2}$ we obtain by a simple computation  
\begin{align} \label{eq:secondbinary}
&\sum_{w_{1}(\operatorname{mod} 2^t)} e\Big(\frac{w_1 u_1^T x_{1}}{2^t} \Big)	\sum_{w_{2} (\operatorname{mod} 2^t)} e\Big(\frac{w_2 (u_2^T x_{2}  + a w_1  2^{\tilde{\nu}})}{2^{t}} \Big)  = 2^{t+\tilde{\nu}} e\Big(\frac{\overline{a \tilde{c}}  \tilde{x}_1 \tilde{x}_2}{2^{t-\nu}}\Big)
\end{align}
provided that $\tilde{x}_1, \tilde{x}_2 \in \Z$.  

By definition of $D$ and $\eta$, it follows that $\alpha(x,Q,\varrho)$ vanishes unless $D^{-1} U^T  x \in \Z^m$. If this holds, we obtain that 
\begin{align*}
\alpha(x,Q,\varrho) = \omega \, \eta(D^{-1} U^T  x)  \frac{\sqrt{\det D}}{\sqrt{\det Q}} 
\end{align*}
for some $\omega$ with  $\abs{\omega}=1$. Substituting $x$ by $U^{-T} D x$ in \eqref{eq:transform} yields the first claim of the lemma. 

The absolute value of the $n$-th Fourier coefficient of $\theta|[\varrho](z) = \sum_n a(n) e(\hat{d} nz/N)$ is given by
\begin{align*}
|a(n)| = \Big|\sum_{\frac{1}{2}x^T S x=n } \alpha(x,Q,\varrho)\Big| \leq 2^{\frac{m}{2}}\frac{ \det D }{\det Q} r(S,n)
\end{align*}
where $S:= \frac{N}{\hat{d}} D \tilde{Q}^{-1}D$. To prove $|a(n)| = r(NQ^{-1}) / \sqrt{\det Q}$ for $d= (N,c)=1$, we need to show that $\alpha(x,Q,\varrho)$ only depends on $n$ and not on the particular choice of $x$. By \eqref{eq:alpha} it suffices to consider the Gauss sum modulo $c$. By \eqref{eq:gaussum}, the dependence on $x$ modulo $\tilde{c}$ is given by 
	\begin{align*}
	e\Big(-\frac{\overline{2^{t+2}a} \sum_{i=1}^m \overline{q_i}  \tilde{x}_i^2}{\tilde{c}} \Big) = e\Big(-\frac{n\overline{2^{2t+2}aN} }{\tilde{c}} \Big).
	\end{align*}
	For the last step, we use that $d_i=1$ and hence,  $\frac{1}{2}U^TQU = \operatorname{diag}(q_1,\ldots,q_m)$ is invertible modulo $\tilde{c}$ which implies that $n = \frac{1}{2} x^T N Q^{-1} x \equiv N \sum_{i=1}^m  \overline{q_i} \tilde{x}_i^2 \,(\operatorname{mod} \tilde{c})$. If $N$ is even, then $c$ must be odd. If $N$ is odd, then $m$ must be even and there are no diagonal terms in the splitting over $\Z_2$. Since the inverse of $(\begin{smallmatrix}
	2 & 1 \\  1 & 2
	\end{smallmatrix})$ modulo $2^t$ is given by $\overline{3} (\begin{smallmatrix}
		2 & -1 \\  -1 & 2
	\end{smallmatrix})$ we obtain that 
	\begin{align*}
	n = \frac{1}{2} x^T N (U^T Q U)^{-1} x =  N \sum_{i=1}^{r_1}  \overline{3} (x_{2i-1}^2 - x_{2i-1} x_{2i} + x_{2i}^2) +  N \sum_{i=1}^{r_1}  x_{2i-1} x_{2i} \,(\operatorname{mod}2^t).
	\end{align*}
	It follows by \eqref{eq:firstbinary} and \eqref{eq:secondbinary} that the dependence on $x$ of the Gauss sum modulo $2^t$ is given by $e\Big(-\frac{n\overline{\tilde{c}aN}}{2^t} \Big)$. 
\end{proof}

This provides the necessary tools to prove our main result: 
\begin{proof}[Proof of Theorem \ref{thm:main} and \ref{thm:main2}]
Both $f$ and $g$ are cusp forms since the value of $\theta(Q,z)$ at each cusp  depends only on a genus-invariant term. Set  $\mu = [\operatorname{SL}_2(\Z) : \Gamma_{0}(N)]$, $\tilde{\mu} = [\operatorname{SL}_2(\Z): \Gamma(N)]$ and recall that
	\begin{align*}
	\frac{\tilde{\mu}}{\mu}  = N^2 \prod_{p \mid N } \Kl{1- \frac{1}{p}}.
	\end{align*}
	We obtain
	\begin{align*}
	\langle g,g \rangle = \int_{\Gamma_{0}(N) \backslash \Ha} \abs{g(z)}^2 y^{\frac{m}{2}-2} dx\,dy = \frac{\mu}{\tilde{\mu}} \sum_{j=1}^{\tilde{\mu}} \int_{\operatorname{SL}_2(\Z) \backslash \Ha} \big|g|[\tau_j])(z)\big|^2 y^{\frac{m}{2}-2} dx\,dy
	\end{align*}
	for a set of coset representatives $\tau_i$ of $\Gamma(N) \backslash \operatorname{SL}_2(\Z)$. On this set, we define an equivalence relation 
	\begin{align*}
	\tau_i \sim \tau_j \Leftrightarrow \tau_i \in \Gamma(N) \tau_j T  \quad \text{for } 
	T= \left. \left\{ \begin{pmatrix}
	1 & x \\ & 1
	\end{pmatrix} \, \right| \, 0 \leq x \leq N-1 \right\}.
	\end{align*}
	For the remaining matrices, a set of  representatives is given by \[\Big\{ \varrho_{a,c}  = \begin{pmatrix} a & * \\ c & * \end{pmatrix} \in \operatorname{SL}_2(\Z) \mid a, c\, (\operatorname{mod} N)\Big\}.\] For simplicity, we set $\varrho := \varrho_{a,c}$. It follows that 
	\begin{align} \label{eq:4.4}
	\langle g,g \rangle = \frac{\mu}{\tilde{\mu}} \sum_{c\, (\operatorname{mod} N)} \sum_{\substack{a (\operatorname{mod} N) \\ (a,c,N)=1}}  \bigcup_{t \in T} \int_{t F}  \big|g|[\varrho](z)\big|^2 y^{\frac{m}{2}-2} dx\,dy,
	\end{align}
	where $F = \{x+iy  \mid - \frac{1}{2} \le x \le \frac{1}{2}, |x+iy| \geq 1 \}$ is a fundamental domain of $\operatorname{SL}_2(\Z) \backslash \Ha$. We put $d=(c,N)$ and write $g|[\varrho] = \sum_n a_{\varrho}(n) e\big(\frac{n \hat{d} z}{N} \big)$ for $\hat{d}$ either $d/2$ or $d$. By the previous display, it follows that 
	\begin{align} \notag
	\langle g, g \rangle &\le N \frac{\mu }{\tilde{\mu}} \sum_{\substack{ c (\operatorname{mod} N) \\ (c,N)=d}} \sum_{\substack{a (\operatorname{mod} N) \\ (a,d)=1}} \int_{\frac{1}{2}}^{\infty} \sum_{n=1}^\infty |a_{\varrho}(n)|^2 \operatorname{exp}\Big(- \frac{2 \pi  y n d}{N}\Big) y^{\frac{m}{2}-2} dy \\  \notag
	&\ll N \frac{\mu }{\tilde{\mu}} \sum_{\substack{ c\, (\operatorname{mod} N) \\ (c,N)=d}} \sum_{\substack{a (\operatorname{mod} N) \\ (a,d)=1}} \Big(\frac{N}{d}\Big)^{\frac{m}{2}-1} \sum_{n \leq (N/d)^{1+\eps}} \frac{|a_{\varrho}(n)|^2}{n^{\frac{m}{2}-1}} + \mathcal{O}(N^{-100}). 
	\end{align}
	For $c=0$, we have $d=N$ and $a_\rho(n) = r(Q,n) - r(Q',n)$. Hence, the contribution of this case is bounded by $\mathcal{O}(N^\eps)$. For $c \neq 0$ we apply Lemma \ref{lemma:trafothetaseries}.  This gives
	\begin{align*}
	|a_\varrho(n)|^2 \le \frac{\det D_\varrho}{\det Q} (r(S_\varrho,n)^2 + r(\tilde{S}_\varrho,n)^2),
	\end{align*}
	for matrices $D_\varrho, S_\varrho,\tilde{S}_\varrho$ given as in Lemma  \ref{lemma:trafothetaseries}. Let $a'_1,\ldots,a'_m$ denote the associated diagonal entries of $S_\varrho$. By Lemma \ref{eq:thetaseriesl} we infer that
	\begin{align*}
	\sum_{n \le (N/d)^{1+\eps}}  \frac{r(S_\varrho,n)^2}{n^{\frac{m}{2}-1}} &\ll \Big( 1+ \sum_{j=m-1}^{2m-2} \Big(\frac{N}{d}\Big)^{\frac{j-m}{2}+1} \frac{\sqrt{a'_1 a'_2}}{\det S_d} \prod_{i=\lfloor \frac{j+5}{2}\rfloor}^m \sqrt{a'_i} \prod_{k=\lfloor \frac{j+4}{2}\rfloor}^m \sqrt{a'_k}\Big) N^\eps.
	\end{align*} 
	To estimate this sum, we use  $a'_1 a'_2 \ll (\det S_\varrho)^{\frac{2}{m}}$ and $a'_i \ll N$ for $3 \le i \le m$ and $\det S_\varrho = (N/d)^m (\det D_\varrho)^2 (\det Q)^{-1}$. This yields that the $j$-th term is bounded by 
	 \begin{align*}
	 N^{-\frac{m}{2}+1} d^{\frac{3m-j}{2}-2} (\det Q)^{1-\frac{1}{m}} (\det D)^{\frac{2}{m}-2}. 
	 \end{align*}
	 If $m$ is even, we use that $a_i \gg 1$ and bound the $j=m-1$ term by 
	\[\frac{\sqrt{N}}{\sqrt{d}\sqrt{\det S_d}} \prod_{i=\frac{m}{2}+2}^{m} \sqrt{a'_i} \ll \frac{d^{\frac{m}{2}} N^{-\frac{m}{4}} \sqrt{\det Q}}{\det D_\varrho}.\]  
By construction, $\det D_\varrho$ only depends on $d=(c,N)$. Since there are $\phi(\frac{N}{d})$ choices for $(c,N)=d$ and $\frac{\phi(d)}{d}N$ for $a$, it follows by the previous five displays that 
\begin{align} \label{eq:est0} 
\langle g,g \rangle \ll N^\eps + \frac{N^\frac{m}{2}}{\det Q}\sum_{d \mid N} \frac{\det D_\varrho}{d^\frac{m}{2}} + \frac{N^{\frac{m}{4}}}{\sqrt{\det Q}}+\frac{N}{ \det Q^\frac{1}{m}} \sum_{d \mid N} \frac{d^{\lfloor\frac{m-3}{2}\rfloor}}{(\det D_\varrho)^{1-\frac{2}{m}}}. 
\end{align}
For the first $d$-sum, we use that
\begin{align} \label{eq:est1} 
\frac{\det D_\varrho }{\det Q} \leq \frac{F(D_\varrho,\frac{m}{2})}{F(Q,\frac{m}{2}) } \le \frac{d^{\frac{m}{2}}}{F(Q,\frac{m}{2})}.
\end{align}
This bound is sharp for at least one $d$ which can be constructed as follows. Let $v_p(j)$ defined as after \eqref{eq:2adicform}, choose $k_p$ maximal such that $v_{p}(k_p) \geq \frac{m}{2}$ and set $d= \prod_{p \mid N} p^{k_p}$. 

By \eqref{eq:est0} and \eqref{eq:est1} we obtain the bound for $m=3$ and $m=4$. For $m \ge 5$, we estimate 
\[
\frac{d^{\frac{m}{2}-\frac{3}{2}}}{(\det D_\varrho)^{1-\frac{2}{m}}} \le  \frac{d^{\frac{m}{2}-\frac{3}{2}}}{F(D_\varrho,\frac{m-1}{2}-\frac{1}{m})^{1-\frac{2}{m}}} \le \frac{N^{\frac{m}{2}-1}}{F (Q,\frac{m-1}{2}-\frac{1}{m})^{1-\frac{1}{m}}}.
\]
By definition, the same bounds hold for $\langle f,f \rangle$. If $Q$ is diagonal, we can do slightly better. The Fourier coefficients $a_\varrho(n)$ of $(f|[\varrho])(z)$ are bounded by \[\frac{\det D_\varrho}{\det Q} (r(S_\varrho,n) + r(\operatorname{gen}S_\varrho,n)).\] For the latter term, we make use of \eqref{eq:genusupperbound}:
\begin{align*} 
\frac{\det D_\varrho}{\det Q} \! \sum_{n \le (N/d)^{1+\eps}} \!  \frac{r(\operatorname{gen} S_\varrho,n)^2}{n^{\frac{m}{2}-1}} \ll \frac{d^m}{N^m \det D_\varrho } \! \sum_{n \le (N/d)^{1+\eps}} \!  n^{\frac{m}{2}-1}  (n,N)   \ll \Big(\frac{N}{d}\Big)^{-\frac{m}{2}} N^{\eps}. 
\end{align*} 
As a result, the contribution of $r(\operatorname{gen}S_\varrho,n)$ is bounded by $\mathcal{O}(N^{\eps})$. 

For $Q = \operatorname{diag}(a_1,\ldots,a_n)$, the diagonal entries of $S_\varrho$  are given by
\[\frac{N (a_1,d)^2}{ d \, a_1}, \ldots, \frac{N (a_m,d)^2}{ d \, a_m}. \]
We sort these entries by value and denote them by  $a'_1,\ldots,a'_m$. For $3 \le l_1,l_2 \le m+1$, we estimate
\begin{align*}
\sqrt{a'_1 a'_2} \prod_{i=l_1}^m \sqrt{a'_i} \prod_{j=l_2}^m \sqrt{a'_j
}  \le 
\Big(\frac{N}{d}\Big)^{\frac{2m+4-l_1-l_2}{2}}  \frac{\det D_\varrho}{\sqrt{a_m a_{m-1}}}.  
\end{align*}
Thus, we get
\begin{align*}
\sum_{d \mid N} \Big(\frac{N}{d}\Big)^{\frac{m}{2}}\frac{\det D_\varrho}{\det Q} \sum_{n \le (N/d)^{1+\eps}}  \frac{r(S_\varrho,n)^2}{n^{\frac{m}{2}-1}} \ll \Big( \frac{N^{\frac{m}{2}}}{\det Q} \sum_{d \mid N} \frac{\det D_\varrho}{d^{\frac{m}{2}}} + \frac{N}{\sqrt{a_m a_{m-1}}} \Big) N^\eps. 
\end{align*}
Applying \eqref{eq:est1} gives the second upper bound of Theorem \ref{thm:main}. 
\end{proof}

It remains to determine a lower bound: 

\begin{thm} Let $f(z) = \theta(Q,z)- \theta(\operatorname{gen}Q,z)$. Then, 
\begin{align*}
\langle f,f \rangle \gg (\min N Q^{-1})^{1-\frac{m}{2}} \frac{N^\frac{m}{2}}{\det Q} + \mathcal{O}(N^\eps). 
\end{align*}
\end{thm}
\begin{proof}
By \eqref{eq:4.4}, we have that 
\begin{align} \label{eq:lowerbound}
\langle f,f \rangle \geq \frac{\mu}{\tilde{\mu}} N \sum_{\substack{c\,(\operatorname{mod}N) \\ d=(c,N)}} \sum_{\substack{a (\operatorname{mod} N) \\ (a,d)=1}} \sum_{n=1}^\infty |a_{\varrho}(n)|^2 \int_{1}^{\infty} \operatorname{exp}\Big(- \frac{4 \pi  y n d}{N}\Big) y^{\frac{m}{2}-2} dy,
\end{align}
where $\varrho= (\begin{smallmatrix}
a & * \\ c & *
\end{smallmatrix})$. For $t>0$, it holds by substitution and \cite[2.33.5]{Gr2007}  that 
\[\int_{1}^{\infty} \operatorname{exp}(- t y) y^{\frac{m}{2}-2} dy \gg t^{1-\frac{m}{2}} \operatorname{exp}(- t).\] 

We drop all terms with $(c,N) > 1$. For the remaining $c$ and $a$, we have  by Lemma \ref{lemma:trafothetaseries} that $|a_\varrho(n)|^2 = (\det Q)^{-1} |r(NQ^{-1},n) - r(\operatorname{gen}NQ^{-1},n)|^2$. Since $\frac{\mu}{\tilde{\mu}} N \varphi(N)  = 1$ it follows that 
\begin{align*}
\langle f,f \rangle \gg \frac{N^{\frac{m}{2}}}{\det Q} \sum_{n=1}^\infty \frac{r(NQ^{-1},n)^2}{n^{\frac{m}{2}-1}} e^{- \frac{4 \pi n}{N}} -  \frac{N^{\frac{m}{2}}}{\det Q}\sum_{n=1}^\infty \frac{r(\operatorname{gen }NQ^{-1},n)^2}{n^{\frac{m}{2}-1}} e^{- \frac{4 \pi n}{N}}.
\end{align*}
By \eqref{eq:genusupperbound} the latter term is bounded by  $\mathcal{O}(N^\eps)$.
 For the former term, we apply Hermite's theorem. This gives 
\[\min N Q^{-1} \leq (4/3)^\frac{m-1}{2} N (\det Q)^{-\frac{1}{m}}.\] 
As a result, we obtain
\begin{align*}
 \sum_{n=1}^\infty \frac{r(NQ^{-1},n)^2}{n^{\frac{m}{2}-1}} \exp\Big(- \frac{ 2 \pi n}{N}\Big) \gg (\min N Q^{-1})^{1-\frac{m}{2}}. 
\end{align*}
This yields the claim. 
\end{proof}

\section{Estimates for the error term}

We start by deriving uniform estimates for the ternary case: 

\begin{lemma} \label{lemma:m=3} Consider a positive, integral, ternary quadratic form $q(x) = \frac{1}{2} x^T Q x$. Then, 
\begin{align*}
|r(Q,n) - r(\operatorname{spn}Q,n)| \ll \frac{\sqrt{N}}{(\det Q)^\frac{1}{6}} \Big(\frac{n^\frac{13}{28}}{N^{\frac{1}{7}}} + \frac{n^\frac{7}{16}}{N^{\frac{1}{16}}} + n^\frac{1}{4} \frac{\sqrt{(\tilde{n},N^\infty)} v^{\frac{1}{4}} \sqrt{(n,N)}}{\sqrt{N} }
	\Big) (nN)^\eps,
\end{align*} 
where $\tilde{n}$ is the largest divisor of $n$ such that $(\tilde{n},N^\infty)$ is squarefree and
\begin{align*}
v= \prod_{\substack{q \textnormal{ anisotropic over } p }} (p^{\infty},N).
\end{align*} 
\end{lemma}
\begin{proof} By a local argument of Blomer \cite[\S\,1.3]{Blomer2008}, it is possible to write 
	\begin{align*}
	r(Q,n) - r(\operatorname{spn}Q,n) = \sum_{j} \gamma_j r(Q_j,m_j v_j w_j) - r(Q'_j,m_j v_j w_j)
	\end{align*} 
	for $Q_j, Q'_j$ in the same spinor genus and of level dividing $N$, $\sum_j |\gamma_j| \ll (nN)^\eps$, $m_j v_j w_j\mid n$ for all $j$, $(m_j,N)=1$, $w_j$ squarefree, $v_j \mid N^2$ and $q$ anisotropic over all prime divisors of $v_j$. The claim now follows by applying Theorem \ref{thm:main} and \eqref{eq:2}. 
\end{proof}

For $m \geq 4$, we determine an effective lower bound for $n$ with respect to $Q$. To this end, recall the definition of $F(Q,s)$ given in \eqref{eq:F(Q,s)}.

\begin{lemma} \label{lemma:m=4} Let $Q$ correspond to a positive, integral, quadratic form in $m$ variables. Then, we have for $m=4$ that 
\begin{align*}
|r(Q,n) - r(\operatorname{gen}Q,n)| \ll n^{\frac{m}{4}-\frac{1}{2}} \Big(\frac{N}{\sqrt{F(Q,2)}} + \frac{\sqrt{N}}{(\det Q)^\frac{1}{8}}\Big)  \min \Big(\sqrt{N}, 1 + \frac{n^\frac{1}{4} (n,N)^\frac{1}{4} }{\sqrt{N}}\Big) (nN)^\eps. 
\end{align*} 
Let $\prod_p \beta_p(q,n) \gg \beta^{-1} (nN)^{-\eps}$. Then, $r(Q,n) \geq 1$ holds if 
\begin{align*}
n \gg \beta^4 (n,N) \Big(N \frac{\det Q}{F(Q)} + (\det Q)^\frac{3}{4} \Big)^2 \quad \text{or} \quad n \gg \beta^2 N^2 \Big(N \frac{\det Q}{F(Q)} + (\det Q)^\frac{3}{4} \Big). 
\end{align*}	
For $m \geq 5$, we have $r(Q,n) \geq 1$ if 	
\begin{align*}
n \gg  \Big( \beta^2 \sqrt{(n,N)} \Big(N^{\frac{m}{2}-1} \frac{\det Q}{F(Q,\frac{m-1}{2}-\frac{1}{m})} + (\det Q)^{1-\frac{2}{m}}  \Big) \Big)^{\frac{2}{m-3}}. 
\end{align*} 
Moreover, we have that 
\begin{align*}
\beta \leq \begin{cases}
\sqrt{(n,N)} & \text{if } m=4 \text{ and } n \text{ primitively locally represented}, \\ 
\min( (n,N), (n,(\det Q)^{\frac{1}{m-4}})) & \text{if } m\geq 5 \text{ and } n \text{  locally represented}. 
\end{cases}
\end{align*}
\end{lemma}
\begin{proof} We apply \eqref{eq:petersson} to bound the Fourier coefficients of $f(z) =\theta(Q,z)- \theta(\operatorname{gen}Q,z)$. For even $m$ this is a result from \cite[Corollary 14.24]{Iw2004}. Since the Weil bound also holds for Kloosterman sums twisted by a quadratic character, cf.\ \cite[Lemma 4]{Wa2017}, this result can be extended to odd $m$. In addition, we have for even $m$ that
\begin{align*}
a(n) \ll \norm{f} n^{\frac{m}{4}-\frac{1}{2}} N^\frac{1}{2} (nN)^\eps
\end{align*}
by applying Deligne's bound, cf.\ \cite[Theorem 11]{RS2018}.
\end{proof}

For $m=4$, this result  implies a substantial saving compared to previous lower bounds such as \cite[Theorem 1]{Ro2019}. For example if $(n,N) =1$ and $\det Q  = F(Q,2)$, a lower bound is given by $n \gg (N^2 + (\det Q)^{4/3})N^\eps$. For $m=6$, we obtain 
\begin{align*}
n \gg (n^2,\det Q)^\frac{2}{3} (n,N)^\frac{1}{3} \Big( N^\frac{2}{3} (\det Q)^\frac{2}{3} + \det Q\Big) N^\eps
\end{align*}
which improves the lower bound $n \gg (\det Q)^{\frac{12}{5}}$ of Hsia and Icaza \cite{HI1999}. 

\section{Applications}

We start with the problem of representing an integer by three squares of almost primes. To this end, we want to sieve the sequence
\begin{align*}
\mathcal{A} = \{x_1 x_2 x_3 \mid   x_1^2 +x_2^2 + x_3^2 = n, x_i \in \Z_{>0}.\},
\end{align*}
where $n \equiv 3\,(\operatorname{mod}24)$, $5 \nmid n$. Let $d$ denote a squarefree integer and $\mathcal{A}_d$  the subset of all $a \in \mathcal{A}$ that are divisible by $d$. Moreover, let $\textbf{l} = (l_1,l_2,l_3)$, $\mu(\textbf{l}) = \mu(l_1) \mu(l_2) \mu(l_3)$ and  $q_\textbf{l}(x) = l_1^2 x_1^2 + l_2^2 x_2^2 + l_3^2 x_3^2$ with corresponding matrix $Q_\textbf{l}$. An application of the inclusion-exclusion principle gives
\begin{align*}
|\mathcal{A}_d| = \mu(d) \sum_{\substack{\textbf{l} \in \N^3 \\ [l_1,l_2,l_3] =d }} \mu(\textbf{l}) r(Q_\textbf{l},n).
\end{align*} 
We decompose $r(Q_\textbf{l},n) = r(\operatorname{gen}Q_\textbf{l},n) + a(n)$. For the main term, we have \[r(\operatorname{gen}Q_\textbf{l},n) = \frac{\pi n^\frac{1}{2}}{4 l_1 l_2 l_3} \prod_p \beta_p(Q_\textbf{l},n).\] 
For $\textbf{1}=(1,1,1)$, we have $\beta_p(Q_\textbf{1},n) \gg n^{-\eps}$ for $n \equiv 3\, (\operatorname{mod} 8)$. Hence, we may set
\[
\omega(\textbf{l},n) := \prod_p \beta_p(Q_\textbf{l},n)  \Big( \prod_p\beta_p(Q_\textbf{1},n)\Big)^{-1}
\] 
and 
\[ \frac{\Omega(d)}{d} = \prod_{p \mid d} \frac{\Omega(p)}{p} := \mu(d) \sum_{\substack{l \in \N^3, [l_1,l_2,l_3]=d}} \mu(l_1) \mu(l_2) \mu(l_3) \frac{\omega(\textbf{l},n)}{l_1 l_2 l_3}.
\]
The functions $\omega(\textbf{l},n)$ are explicitly computed in \cite[Lemma 3.1 \& 3.2]{BB2005}. This yields for $n \equiv 3 \,(\operatorname{mod} 24)$, $5 \nmid n$  that 
$0 \leq \Omega(p) <p$. Furthermore, there is constant $A \ge 2$ such that 
\begin{align*}
\prod_{z_1 \le p < z} \Big(1- \frac{\Omega(p)}{p} \Big)^{-1} \le \Big(\frac{\log z}{\log z_1}\Big)^3 \Big(1+ \frac{A}{{\log z_1}}\Big) \text{ for } 2 \le z_1 \le z.
\end{align*}
For $X := \frac{\pi}{4} n^\frac{1}{2} \prod_p \beta_p(Q_\textbf{1},n)$, we obtain that $|\mathcal{A}_d| = \frac{\Omega(d)}{d} X + R_d(\mathcal{A})$, where 
\begin{align*}
R_d(\mathcal{A}) = \mu(d)\sum_{\substack{l \in \N^3, [l_1,l_2,l_3]=d}} \mu(\textbf{l}) (r(Q_\textbf{l},n) - r(\operatorname{gen}Q_\textbf{l},n)).
\end{align*} 
Let $F_3(s)$ and $f_3(s)$ denote the three dimensional sieve functions given in \cite[Theorem 0]{DHR1997}. Then, it holds by \cite[Theorem 1]{DHR1997} for $\kappa =3$ that: 
\begin{lemma} Let $\mathcal{A}, X, \Omega$ as above and write $v(d)$ for the number of prime factors of $d$. Furthermore, let
\begin{align*}
\sum_{d \leq X^\tau} \mu^2(d) 4^{v(d)} |R_d(\mathcal{A})| \ll \frac{X}{(\log X)^4}	
\end{align*}
for $\tau \in \R$ with $0< \tau <1$. Then, it holds for any $u,v \in \R$ with $1/\tau <u <v$ and $\tau v > \beta_3 := 6.6408$ that 
\begin{align*}
|\{P_r \mid P_r \in \mathcal{A}\}| \gg X \prod_{p < X^\frac{1}{v}} \Big(1 - \frac{\Omega(p)}{p}\Big)
\end{align*}
provided only that 
\begin{align} \label{eq:hard}
r > 3 u -1 + \frac{3}{f_3(\tau v)} \int_{1}^{v/u} F_3(\tau v-s) \Big(1- \frac{u}{v}s \Big) \frac{ds}{s}.
\end{align}
\end{lemma}

We prove the following improvement of \cite[Lemma 4.1]{L2007}: 
 
\begin{lemma} Let $n \equiv 3\, (\operatorname{mod} 24), 5 \nmid n$ and $\tau < \frac{3}{58}$. Then, 
\begin{align*}
\sum_{d \leq n^\frac{\tau}{2}}  \mu^2(d) 4^{v(d)} |R_d(\mathcal{A})| \ll n^{\frac{1}{2}-\eps}.
\end{align*}
\end{lemma} 
\begin{proof}
We only consider combinations of  $l_1,l_2,l_3$ such that either $r(Q_\textbf{l},n) \neq 0$ or  $r(\operatorname{gen}Q_\textbf{l},n) \neq 0$. The remaining forms are isotropic over all odd primes and satisfy $2 \nmid l_1l_2l_3$. It follows by \eqref{eq:gen=spn} that $r(\operatorname{gen}Q_\textbf{l},n) = r(\operatorname{spn}Q_\textbf{l},n)$. Since the level of $Q_\textbf{l}$ is given by $4d$ we infer by Lemma \ref{lemma:m=3} that
\begin{align*}
&\sum_{d \leq n^{\tau/2}} \mu^2(d) \sum_{ [l_1,l_2,l_3]=d} \mu^2(\textbf{l}) 4^{v(d)} |r(Q_\textbf{l},n) - r(\operatorname{gen}Q_\textbf{l},n)| \\ &\ll \sum_{d \leq n^{\tau/2}}  d^{\frac{2}{3}} \Big( \frac{n^{\frac{13}{28}}}{d^{\frac{2}{7}}} + \frac{n^{\frac{7}{16}}}{d^{\frac{1}{8}}} + n^{\frac{1}{4}} \frac{\sqrt{(n,d)(n,d^2)}}{d} \Big) d^\eps,
\end{align*}
where we estimated $\#\{ \textbf{l} \in \N^3 \mid [l_1,l_2,l_3] =d\} \le \tau(d)^3 \ll d^\eps$ in the second step.  For $\tau <\frac{3}{58}$ the display above is bounded by $\ll n^{\frac{1}{2}-\eps}$. 
\end{proof}

To avoid the computation of $F_3(s)$ and $f_3(s)$, we apply a well known trick. If we set $\tau u = 1 + \zeta - \zeta/ \beta_3$ and $\tau v = \beta_3 / \zeta + \beta_k -1$, it follows by (10.1.10), (10.2.4) and (10.2.7) in \cite{HR1974} that 
\begin{align*}
\frac{3}{f_k(\tau v)} \int_{1}^{v/u} F_k(\tau v-s) \Big(1- \frac{u}{v}s \Big) \frac{ds}{s} \le (3+ \zeta) \log \frac{\beta_3}{\zeta} -3 + \zeta \frac{3}{\beta_3}. 
\end{align*}
This gives the following slightly weaker version of \eqref{eq:hard}:
\begin{align*}
r > \frac 3 \tau (1+ \zeta) -1 + (3+\zeta) \log \frac{\beta_3}{\zeta} -3 - \zeta \frac{3 ( \frac 1 \tau -1)}{\beta_3} := m(\zeta).
\end{align*}
For $\tau = 3/58$ it follows that 
\begin{align*}
r > \min_{0 < \zeta < \beta_3} m(\zeta) = m(0.0560831\ldots) \approx 71.3875. 
\end{align*}
This proves Corollary \ref{cor:m=3}. 

Even for four squares of primes, current technology is not sufficient to prove that every  $n \equiv 4\,(\operatorname{mod}24)$ is represented. By means of sieving and Theorem \ref{thm:main}, we can address this problem for almost primes. Here, the main input is: 
 
\begin{lemma} Let $n \equiv 4\, (\operatorname{mod} 24)$, $\tau < \frac{1}{2}$, $q_\textbf{l}(x) = l_1^2 x_1^2 + l_2^2x_2^2+l_3^2 x_3^2 + x_4^2 l_4^2$. Then, 
	\begin{align*}
	\sum_{d \leq n^\frac{\tau}{2}} \mu(d) \sum_{\substack{ l \in \N^4 \\ d=[l_1,l_2,l_3,l_4]}} \mu(\textbf{l}) 4^{v(d)} |r(Q_\textbf{l},n) - r(\operatorname{gen}Q_\textbf{l},d)| \ll n^{\frac{1}{2}-\eps}.
	\end{align*}
\end{lemma} 

By applying the  weighted four dimensional sieve from \cite{DHR1997} it follows that every $n\equiv 4\, (\operatorname{mod} 24)$ is represented by $x_1^2 + x_2^2 + x_3^2 + x_4^2$ with $x_1 x_2 x_3 x_4 \in P_{20}$. However, a recent approach by Tsang, Zhao \cite{Tsang2017} and Ching \cite{Ching2018} yields much better results. Their idea is to choose one of the $x_i$'s as a prime and then to combine Chen's switching result with a three dimensional sieve. This shows that $n = p^2 + x_1^2 + x_2^2 + x_3^2$ for a prime $p$ and $x_1x_2x_3 \in P_{12}$. 

To tackle the problem of representing an integer by three squares of smooth numbers, we follow the approach in \cite{BBD2009}. The underlying idea is to choose integers $d_1,d_2,d_3 \in [n^{\eta}, 2 n^{\eta}]$ with $\eta$ as large as possible such that 
\begin{align*}
q(x) = d_1^2 x_1^2 + d_2^2 x_2^2 + d_3^2 x_3^2 = n
\end{align*}
is soluble. To obtain a lower bound for $r(\operatorname{spn}Q,n)$, we choose $d_1,d_2,d_3$ to be distinct primes  that are  $\equiv 1 \,(\operatorname{mod} 4)$ and coprime to $n$. This implies that
\begin{align} \label{eq:upper}
r(\operatorname{spn}Q,n) = r(\operatorname{gen}Q,n)  \gg \frac{n^{\frac{1}{2}}}{{d_1 d_2 d_3}} (nN)^{-\eps} \gg n^{\frac 1 2-3\eta - \eps}.
\end{align}
Furthermore, we obtain by Lemma \ref{lemma:m=3} that 
\begin{align} \label{eq:lower}
| r(Q,n) - r(\operatorname{spn}Q,n)| &\ll (d_1d_2d_3)^\frac{2}{3} \Big(\frac{n^{\frac{13}{28}}}{(d_1d_2d_3)^{2/7}} + \frac{n^{\frac{7}{16}}}{(d_1d_2d_3)^{1/8}} \Big) n^\eps \\ \notag &\ll \Big(n^{\frac{13}{28}+ \frac{8}{7} \eta }+ n^{\frac{7}{16}+ \frac{13}{8} \eta }\Big) n^{ \eps}. 
\end{align}  
Hence, for $\eta = \frac{1}{116}-\eps$, equation \eqref{eq:lower} is dominated by \eqref{eq:upper} and it follows that $r(Q,n) \geq 1$ for $n$ sufficiently large. This proves the first claim of Corollary \ref{cor:smooth}. 

To derive results for sums of four squares, we make use of the distribution of smooth numbers in short intervals. For every $n \in \N$ it is possible to choose $x$  with largest prime factor not exceeding $n^\frac{1}{4}$ such that $0 \leq n -x^2 \leq n^\frac{5}{8}$ and ${n -x^2 \not\equiv 0,1,4,7\,(\operatorname{mod} 8)}$, see \cite[\S\,5]{BBD2009} for details. By our previous result for three squares, this implies that every  sufficiently large $n$ is represented by the sum of four squares whose largest prime divisors do not exceed $n^\frac{\theta}{2}$ where $\theta = \frac{285}{464}$.

The bounds for three smooth squares can be improved, if we choose $d_1 = e_1 e_2$, ${d_2 = e_1 e_3}$ and $d_3 = e_2 e_3$ for $e_1,e_2,e_3$ mutually coprime and $(e_1e_2e_3,n)=1$. However, this only works, if $\big(\frac{p}{n} \big) =1$ for every prime divisor $p$ of $e_1 e_2 e_3$, since then $\beta_p(n,Q) = 1 + \big(\frac{p}{n} \big) =2$. If this latter condition is satisfied, we can choose $\eta = \frac{1}{80}- \eps$.


\bibliographystyle{amsplain}
\bibliography{mybibfabri44018}

\providecommand{\bysame}{\leavevmode\hbox to3em{\hrulefill}\thinspace}
\providecommand{\MR}{\relax\ifhmode\unskip\space\fi MR }
\providecommand{\MRhref}[2]{%
  \href{http://www.ams.org/mathscinet-getitem?mr=#1}{#2}
}
\providecommand{\href}[2]{#2}
\begin{thebibliography}{10}

\bibitem{Bl2004}
V.~Blomer, \emph{Uniform bounds for {Fourier} coefficients of theta-series with
  arithmetic applications}, Acta Arithmetica \textbf{114} (2004), no.~1, 1--21.

\bibitem{Blomer2008}
\bysame, \emph{Ternary quadratic forms, and sums of three squares with
  restricted variables}, {CRM} Proceedings and Lecture Notes, American
  Mathematical Society, 2008, pp.~1--17.

\bibitem{BB2005}
V.~Blomer and J.~Brüdern, \emph{A three squares theorem with almost primes},
  Bulletin of the London Mathematical Society \textbf{37} (2005), no.~4,
  507--513.

\bibitem{BBD2009}
V.~Blomer, J.~Brüdern, and R.~Dietmann, \emph{Sums of smooth squares},
  Compositio Mathematica \textbf{145} (2009), no.~6, 1401–1441.

\bibitem{Browning2007}
T.~D. Browning and R.~Dietmann, \emph{On the representation of integers by
  quadratic forms}, Proceedings of the London Mathematical Society \textbf{96}
  (2007), no.~2, 389--416.

\bibitem{Cai2012}
Y.~Cai, \emph{Gauss{\textquotesingle}s three squares theorem involving
  almost-primes}, Rocky Mountain Journal of Mathematics \textbf{42} (2012),
  no.~4, 1115--1134.

\bibitem{Ca87}
J.W.S. Cassels, \emph{Rational quadratic forms}, London Mathematical Society
  Monographs, vol.~13, Academic Press, 1978.

\bibitem{Ching2018}
T.~W. Ching, \emph{Lagrange{\textquotesingle}s equation with one prime and
  three almost-primes}, Journal of Number Theory \textbf{183} (2018), 442--465.

\bibitem{DHR1997}
H.~Diamond and H.~Halberstam, \emph{Some applications of sieves of dimension
  exceeding 1}, Sieve Methods, Exponential Sums, and their Applications in
  Number Theory, Cambridge University Press, 1997, pp.~101--108.

\bibitem{Duke2005}
W.~Duke, \emph{On ternary quadratic forms}, Journal of Number Theory
  \textbf{110}, no.~1, 37--43.

\bibitem{Fomenko1991}
O.~M. Fomenko, \emph{Estimates of {Petersson}{\textquotesingle}s inner squares
  of cusp forms and arithmetic applications}, Journal of Soviet Mathematics
  \textbf{53} (1991), no.~3, 323--338.

\bibitem{Gr2007}
I.~S. Gradshteyn and I.~M. Ryzhik, \emph{Table of integrals, series, and
  products}, 7th ed., Academic Press,New York, 2007.

\bibitem{HR1974}
H.~Halberstam and H.-E. Richert, \emph{Sieve methods}, Academic Press, New
  York, 1974.

\bibitem{Ha2004}
J.~Hanke, \emph{Local densities and explicit bounds for representability by a
  quadratic form}, Duke Mathematical Journal \textbf{124} (2004), no.~2,
  351--388.

\bibitem{HI1999}
J.~Hsia and M.~Icaza, \emph{Effective version of
  {Tartakowsky}{\textquotesingle}s theorem}, Acta Arithmetica \textbf{89}
  (1999), no.~3, 235--253.

\bibitem{Iw2004}
H.~Iwaniec and E.~Kowalski, \emph{Analytic number theory}, Colloquium
  Publications, vol.~53, American Mathematical Society, 2004.

\bibitem{L2007}
G.~L\"{u}, \emph{Gauss{\textquotesingle}s three squares theorem with almost
  prime variables}, Acta Arithmetica \textbf{128} (2007), no.~4, 391--399.

\bibitem{Me}
O.T. O'Meara, \emph{Introduction to quadratic forms}, Grundlehren Math. Wiss.,
  vol. 117, Springer, New York, 1973.

\bibitem{Rouse2014}
J.~Rouse, \emph{Quadratic forms representing all odd positive integers},
  American Journal of Mathematics \textbf{136} (2014), no.~6, 1693--1745.

\bibitem{Ro2019}
\bysame, \emph{Integers represented by positive-definite quadratic forms and
  {Petersson} inner products}, Acta Arithmetica \textbf{187} (2019), 81--100.

\bibitem{Sa2018}
N.~T. Sardari, \emph{Quadratic forms and semiclassical eigenfunction hypothesis
  for flat tori}, Communications in Mathematical Physics \textbf{358} (2018),
  895--917.

\bibitem{SP2001}
R.~Schulze-Pillot, \emph{On explicit versions of
  {Tartakovski}{\textquotesingle}s theorem}, Archiv der Mathematik \textbf{77}
  (2001), no.~2, 129--137.

\bibitem{RS2018}
R.~Schulze-Pillot and A.~Yenirce, \emph{Petersson products of bases of spaces
  of cusp forms and estimates for {Fourier} coefficients}, International
  Journal of Number Theory \textbf{14} (2018), no.~8, 2277--2290.

\bibitem{Si1935}
C.~L. Siegel, \emph{{\"Uber} die analytische {Theorie} der quadratischen
  {Formen}}, Annals of Mathematics \textbf{36} (1935), no.~3, 527--600.

\bibitem{Ta29}
W.~Tartakowsky, \emph{Die {Gesamtheit} der {Zahlen}, die durch eine positive
  quadratische {Form }$f(x_1, . . . , x_s) (s \geq 4)$ darstellbar sind}, Izv.
  Akad. Nauk SSSR \textbf{7} (1929), no.~1, 111--121.

\bibitem{Tsang2017}
K.-M. Tsang and L.~Zhao, \emph{On {Lagrange}'s four squares theorem with almost
  prime variables}, Journal f\"{u}r die reine und angewandte Mathematik
  (Crelles Journal) \textbf{2017} (2017), no.~726.

\bibitem{Wa2017}
F.~Waibel, \emph{Fourier coefficients of half-integral weight cusp forms and
  {Waring's} problem}, The Ramanujan Journal \textbf{47} (2017), no.~1,
  185--200.

\bibitem{Wa1960}
G.~L. Watson, \emph{Integral quadratic forms}, Cambridge University Press, New
  York, vol.~51, Cambridge Tracts in Mathematics and Mathematical Physics,
  1960.

\bibitem{Ya1998}
T.~Yang, \emph{An explicit formula for local densities of quadratic forms},
  Journal of Number Theory \textbf{72} (1998), no.~2, 309--356.

\end{thebibliography}
\end{document}